\documentclass[smallextended,nospthms,envcountsect]{svjour3}

\smartqed 

\usepackage{graphicx}
\usepackage{mathptmx}
\usepackage{amssymb}
\usepackage{amsmath}
\usepackage[matrix,arrow,curve,cmtip]{xy}
\usepackage{mathrsfs}
\usepackage{amsthm}

\numberwithin{equation}{section}

\theoremstyle{plain}   %% This is the default, anyway
\newtheorem{bigthm}{Theorem}   % Numbered separately, as A, B, etc.
\newtheorem{theorem}[equation]{Theorem}  % Numbered with the equation counter
     
\newtheorem{lemma}[equation]{Lemma}         
\newtheorem{prop}[equation]{Proposition}

\theoremstyle{definition}

\theoremstyle{remark}
\newtheorem{remark}[equation]{Remark}

\newcommand{\Tor}{\operatorname{Tor}}

\newcommand{\HH}{\operatorname{HH}}

\newcommand{\Z}{\mathbb{Z}}

\newcommand{\colim}{\operatornamewithlimits{colim}}

\newcommand{\pr}{\operatorname{pr}}

\newcommand{\xto}{\xrightarrow}

\journalname{Mathematische Annalen}

\begin{document}

\title{On a conjecture of Vorst\thanks{Both authors were supported in
part by grants from the National Science Foundation. The first author
received additional support from the Japan Society for the Promotion
of Science.}}

\author{Thomas Geisser \and Lars Hesselholt}

\institute{Thomas Geisser \at
University of Southern California, Los Angeles, California\\
\email{geisser@usc.edu}
\and
Lars Hesselholt \at
Nagoya University, Nagoya, Japan \\
\email{larsh@math.nagoya-u.ac.jp}
}

\date{}

\maketitle

%\begin{abstract}
%\keywords{Negative $K$-groups \and topological cyclic homology \and
%$\cdh$-topology} 
%\subclass{19D35 \and 14F20 \and 19D55}
%\end{abstract}

\section*{Introduction}

A ring $R$ is defined to be $K_n$-regular, if the map $K_n(R) \to
K_n(R[t_1,\dots,t_r])$ induced by the canonical inclusion is an
isomorphism for all $r \geqslant 0$~\cite[Definition~2.2]{bass1}. It
was proved by Quillen~\cite[Corollary of Theorem~8]{quillen} that a
(left) regular noetherian ring is $K_n$-regular for all integers
$n$. A conjecture of Vorst~\cite[Conjecture]{vorst} predicts that,
conversely, if $R$ is a commutative ring of dimension $d$ essentially
of finite type over a field $k$, then $K_{d+1}$-regularity implies
regularity. Recently, Corti\~{n}as, Haesemeyer, and
Weibel showed that the  conjecture holds, if the field $k$ has
characteristic zero~\cite[Theorem~0.1]{cortinashaesemeyerweibel}.
In this  paper, we prove the following slightly weaker result, if $k$ 
is an infinite perfect field of characteristic $p > 0$ and strong
resolution of singularities holds over $k$ in the sense of
Section~\ref{ktheorysection} below. 

\begin{bigthm}\label{maintheorem}Let $k$ be an infinite perfect field of
characteristic $p > 0$ such that strong resolution of singularities
holds over $k$. Let $R$ be a localization of a $d$-dimensional
commutative $k$-algebra of finite type and suppose that $R$ is
$K_{d+1}$-regular. Then $R$ is a regular ring.
\end{bigthm}

We also prove a number of results for more general fields of
characteristic $p > 0$. For instance, we show in
Theorem~\ref{maintheoremplus} that, if strong resolution of
singularities holds over all infinite perfect fields of characteristic
$p$, then for every field $k$ that contains an infinite perfect
subfield of characteristic $p$ and every $k$-algebra $R$ essentially
of finite type, $K_q$-regularity for all $q$ implies regularity.

We give a brief outline of the proof of Theorem~\ref{maintheorem}. Let
$\mathfrak{m} \subset R$ be a maximal ideal, and let $d_{\mathfrak{m}}
= \dim(R_{\mathfrak{m}})$ and $e_{\mathfrak{m}} =
\dim_{R/\mathfrak{m}}(\mathfrak{m}/\mathfrak{m}^2)$ be the dimension
and embedding dimension, respectively. One always has
$d_{\mathfrak{m}} \leqslant e_{\mathfrak{m}}$ and the ring $R$ is said
to be regular if $d_{\mathfrak{m}} = e_{\mathfrak{m}}$ for every
maximal ideal $\mathfrak{m} \subset R$. Now, we show in
Theorem~\ref{ktheorem} below that if $R$ is $K_{d+1}$-regular, then the group $K_{d_{\mathfrak{m}}+1}(R_{\mathfrak{m}})/pK_{d_{\mathfrak{m}}+1}(R_{\mathfrak{m}})$
is zero for every maximal ideal $\mathfrak{m}
\subset R$. We further show in Theorem~\ref{hhtheorem} below that for every
maximal ideal $\mathfrak{m} \subset R$, the group
$K_q(R_{\mathfrak{m}})/pK_q(R_{\mathfrak{m}})$ is non-zero for all $0
\leqslant q \leqslant e_{\mathfrak{m}}$. Together the two theorems
show that $d_{\mathfrak{m}} \geqslant e_{\mathfrak{m}}$ as
desired. Theorem~\ref{maintheorem} follows. 

\section{$K$-theory}\label{ktheorysection}

In this section, we prove Theorem~\ref{ktheorem} below. We say that strong
resolution of singularities holds over the (necessarily perfect) field
$k$ if for every integral scheme $X$ separated and of finite type
over $k$, there exists a sequence of blow-ups
$$X_n \to X_{n-1} \to \dots \to X_1 \to X_0 = X$$
such that the reduced scheme $X_r^{\operatorname{red}}$ is smooth over
$k$; the center $Y_i$ of of the blow-up $X_{i+1} \to X_i$ is connected
and smooth over $k$; the closed embedding of $Y_i$ in $X_i$ is
normally flat; and $Y_i$ is nowhere dense in $X_i$.

\begin{prop}\label{KHproposition}Let $k$ be an infinite perfect field
of characteristic $p > 0$ and assume that strong resolution of
singularities holds over $k$. Let $X$ be the limit of a cofiltered
diagram $\{ X_i \}$ with affine transition maps of $d$-dimensional
schemes separated and of finite type over $k$. Then $KH_q(X,\Z/p\Z)$
vanishes, for $q > d$. 
\end{prop}

\begin{proof}It follows from~\cite[Sect.~IV.8.5]{ega} that for all
integers $q$, the canonical map
$$\colim_i K_q(X_i,\Z/p\Z) \to K_q(X,\Z/p\Z)$$
is an isomorphism. Therefore, using the natural spectral sequence
$$E_{s,t}^1 = N_sK_t(U,\Z/p\Z) \Rightarrow KH_{s+t}(U,\Z/p\Z),$$
we conclude that for all integers $q$, the canonical map
$$\colim_i KH_q(X_i,\Z/p\Z) \to KH_q(X,\Z/p\Z)$$
is an isomorphism. Hence, we may assume that $X$ itself is a
$d$-dimensional scheme separated and of finite type over
$k$. In fact, we may even assume that $X$ is integral. Indeed, it
follows from~\cite[Theorem~2.3]{weibel4} that for all integers $q$, the
canonical map
$$KH_q(X,\Z/p\Z) \to KH_q(X^{\operatorname{red}},\Z/p\Z)$$
is an isomorphism, so we may assume that $X$ is reduced. Moreover, if
$X_1 \subset X$ is an irreducible component and $X_2 \subset X$ the
closure of $X \smallsetminus  X_1$, then $X_{12} = X_1 \cap X_2$ has
smaller dimension than $X$ and by~\cite[Corollary~4.10]{weibel4} there
is a long exact sequence
$$\cdots \to
KH_q(X,\Z/p\Z) \to
KH_q(X_1,\Z/p\Z) \oplus KH_q(X_2,\Z/p\Z) \to
KH_q(X_{12},\Z/p\Z) \to \cdots$$
Therefore, a downward induction on the number of irreducible
components shows that we can assume $X$ to be integral. So we let $X$
be integral and proceed by induction on $d \geqslant 0$. In the case
$d = 0$, $X$ is a finite disjoint union of prime spectra of fields
$k_{\alpha}$ with $[k_{\alpha} \colon k] < \infty$. It follows that
the canonical maps 
$$KH_q(X,\Z/p\Z) \leftarrow
K_q(X,\Z/p\Z) \to
\prod_{\alpha} K_q(k_{\alpha},\Z/p\Z)$$
are isomorphisms, and since the fields $k_{\alpha}$ again are perfect
of characteristic $p > 0$, the right-hand group is zero, for $q > 0$
as desired~\cite{kratzer}. So we let $d \geqslant 1$ and assume that
the statement has been proved for smaller $d$. By the assumption that
resolution of singularities holds over $k$, there exists a proper
bi-rational morphism $X' \to X$ from a scheme $X'$ smooth over $k$. We
may further assume that $X'$ is of dimension $d$. We choose a closed
subscheme $Y$ of $X$ that has dimension at most $d-1$ and contains the
singular set of $X$ and consider the cartesian square 
$$\xymatrix{
{ Y' } \ar[r] \ar[d] &
{ X' } \ar[d]<-.2ex> \cr
{ Y } \ar[r] &
{ X. } \cr
}$$
Since the field $k$ is assumed to be an infinite perfect field such
that strong resolution of singularities holds over $k$, the proof
of~\cite[Theorem~3.5]{haesemeyer} shows that the cartesian square above
induces a long exact sequence
$$\cdots \to KH_q(X,\Z/p\Z) \to KH_q(X',\Z/p\Z) \oplus KH_q(Y,\Z/p\Z)
\to KH_q(Y',\Z/p\Z) \to \cdots$$
Now, the schemes $Y$ and $Y'$ are of dimension at most $d-1$ and are
separated and of finite type over $k$. Therefore, the groups
$KH_q(Y,\Z/p\Z)$ and $KH_q(Y',\Z/p\Z)$ vanish, for $q > d-1$, by the
inductive hypothesis. Finally, since the scheme $X'$ is smooth over
$k$, the canonical map defines an isomorphism
$$K_q(X',\Z/p\Z) \xto{\sim} KH_q(X',\Z/p\Z),$$
and by~\cite[Theorem~8.4]{geisserlevine} the common group vanishes for
$q > d$. We conclude from the long exact sequence that
$KH_q(X,\Z/p\Z)$ vanishes, for $q > d$, as desired.
\end{proof} 

\begin{theorem}\label{ktheorem}Let $k$ be an infinite perfect field of
positive characteristic $p$ such that strong resolution of
singularities holds over $k$. Let $R$ be a localization of a
$d$-dimensional $k$-algebra of finite type and assume that $R$ is
$K_{d+1}$-regular. Then the  group $K_{d+1}(R)/pK_{d+1}(R)$ is zero.
\end{theorem}

\begin{proof}Since we assume $R$ is $K_{d+1}$-regular, a theorem of
Vorst~\cite[Corollary~2.1]{vorst} shows that $R$ is $K_q$-regular for all
$q \leqslant d+1$, or equivalently, that the groups $N_sK_q(R)$
vanish for all $s > 0$ and $q \leqslant d+1$. The coefficient exact
sequence 
$$0 \to N_sK_q(R)/pN_sK_q(R) \to N_sK_q(R,\Z/p\Z) \to 
\operatorname{Tor}_1^{\Z}(N_sK_{q-1}(R),\Z/p\Z) \to 0$$
then shows that the groups $N_sK_q(R,\Z/p\Z)$ vanish for $s > 0$ and
$q \leqslant d+1$. Therefore, we conclude from the spectral sequence
$$E_{s,t}^1 = N_sK_t(R,\Z/p\Z) \Rightarrow KH_{s+t}(R,\Z/p\Z)$$
that the canonical map
$$K_q(R,\Z/p\Z) \to KH_q(R,\Z/p\Z)$$
is an isomorphism for $q \leqslant d+1$. Now, for $q = d+1$,
Proposition~\ref{KHproposition} shows that the common group is zero, and
hence, the coefficient sequence
$$0 \to K_{d+1}(R)/pK_{d+1}(R) \to K_{d+1}(R,\Z/p\Z) \to
\Tor(K_d(R),\Z/p\Z) \to 0$$
shows that the group $K_{d+1}(R)/pK_{d+1}(R)$ is zero as stated.
\end{proof}

\section{Hochschild homology}\label{hochschildhomologysection}

In this section, we prove the following general result.

\begin{theorem}\label{hhtheorem}Let $\kappa$ be a commutative ring,
let $r$ be a positive integer, and let $A$ be the $\kappa$-algebra
$A = \kappa[x_1,\dots,x_r]/(x_ix_j \mid 1 \leqslant i \leqslant j
\leqslant r)$. Then, for all $1 \leqslant q \leqslant r$, the image of
the symbol $\{1+x_1, \dots, 1+x_q\}$ by the composition
$$K_q(A) \to \HH_q(A) \to \HH_q(A/\kappa)$$
of the Dennis trace map and the canonical map from absolute Hochschild
homology to Hochschild homology relative to the ground ring $\kappa$
is non-trivial.
\end{theorem}

To prove Theorem~\ref{hhtheorem}, we first evaluate the groups
$\HH_*(A/\kappa)$ that are target of the map of the statement. By
definition, these are the homology groups of the chain complex
associated with the cyclic $\kappa$-module $\HH(A/\kappa)[-]$ defined
by
$$\HH(A)[n] = A \otimes_{\kappa} \dots \otimes_{\kappa} A \hskip6mm 
\text{($n+1$ factors)}$$
with cyclic structure maps
$$\begin{aligned}
d_i & \colon \HH(A/\kappa)[n] \to \HH(A/\kappa)[n-1] \hskip6mm
(0 \leqslant i \leqslant n) \cr
s_i & \colon \HH(A/\kappa)[n] \to \HH(A/\kappa)[n+1] \hskip6mm
(0 \leqslant i \leqslant n) \cr
t_n & \colon \HH(A/\kappa)[n] \to \HH(A/\kappa)[n] \cr
\end{aligned}$$
defined by
$$\begin{aligned}
d_i(a_0 \otimes \dots \otimes a_n) & = \begin{cases}
a_0 \otimes \dots \otimes a_ia_{i+1} \otimes \dots \otimes a_n & \hskip5mm
(0 \leqslant i < n) \cr
a_na_0 \otimes a_1 \otimes \dots \otimes a_{n-1} & \hskip5mm
(i = n) \cr
\end{cases} \cr
s_i(a_0 \otimes \dots \otimes a_n) & = a_0 \otimes \dots \otimes a_i \otimes 1 \otimes a_{i+1} \otimes \dots \otimes a_n \cr
t_n(a_0 \otimes \dots \otimes a_n) & = a_n \otimes a_0 \otimes a_1 \otimes \dots \otimes a_{n-1}. \cr
\end{aligned}$$
The cyclic $\kappa$-module $\HH(A/\kappa)[-]$ admits a direct sum
decomposition as follows. Recall that a word of length
$m$ with letters in a set $S$ is defined to be a function
$$\omega \colon \{1,2, \dots, m\} \to S.$$
The cyclic group $C_m$ of order $m$ acts on the set $\{1, 2, \dots,
m\}$ by cyclic permutation of the elements. We define a cyclical word
of length $m$ with letters in $S$ to be an orbit for the induced
action on the set of words of length $m$ with letters in $S$. We write
$[\omega]$ for the orbit through $\omega$ and call the length of the
orbit the period of $[\omega]$. In particular, the set that consists
of the empty word is a cyclical word $[0]$ of length $0$ and period
$1$. Then the cyclic $\kappa$-module $\HH(A/\kappa)[-]$ decomposes as
the direct sum
$$\HH(A/\kappa)[-] = \bigoplus_{[\omega]} \HH(A/\kappa;[\omega])[-],$$
where the direct sum ranges over all cyclical words with letters in
$\{x_1,\dots,x_r\}$, where the summand $\HH(A/\kappa;[0])[-]$ is the
sub-cyclic $\kappa$-module generated by the $0$-simplex $1$, and where
the summand $\HH(A/\kappa;[\omega])[-]$ with $\omega =
(x_{i_1},\dots,x_{i_m})$, $m \geqslant 1$, is the sub-cyclic
$\kappa$-module generated by the $(m-1)$-simplex $x_{i_1} \otimes
\dots \otimes x_{i_{m}}$. 

\begin{lemma}\label{summandhomology}Let $\kappa$ be a commutative ring,
let $r$ be a positive integer, and let $A$ be the $\kappa$-algebra
$A = \kappa[x_1,\dots,x_r]/(x_ix_j \mid 1 \leqslant i \leqslant j
\leqslant r)$. Let $\omega = (x_{i_1},\dots,x_{i_m})$ be a word with
letters in $\{x_1,\dots,x_r\}$ of length $m \geqslant 0$ and period
$\ell \geqslant 1$.
\begin{enumerate}
\item[(1)] If $m = 0$, then $\HH_0(A/\kappa;[\omega])$ is the free
  $\kappa$-module of rank one generated by the class of the cycle $1$
  and the remaining homology groups are zero.  
\item[(2)] If $m$ is odd or $\ell$ is even, then
  $\HH_{m-1}(A/\kappa;[\omega])$ and $\HH_m(A/\kappa;[\omega])$ are
  free $\kappa$-modules of rank one generated by the classes of the
  cycles $x_{i_1} \otimes \dots \otimes x_{i_m}$ and 
  $\sum_{0 \leqslant u < \ell}
  (-1)^{(m-1)u}t_ms_{m-1}t_{m-1}^u(x_{i_1} \otimes \dots \otimes
  x_{i_m})$, respectively, and the remaining homology groups are zero. 
\item[(3)] If $m \geqslant 2$ is even and $\ell$ is odd, then
  $\HH_{m-1}(A/\kappa;[\omega])$ is isomorphic to $\kappa/2\kappa$
  generated by the class of the cycle $x_{i_1} \otimes \dots \otimes
  x_{i_m}$, there is an isomorphism of the $2$-torsion
  sub-$\kappa$-module $\kappa[2] \subset \kappa$ onto
  $\HH_m(A/\kappa;[\omega])$ that takes $a \in \kappa[2]$ to the class
  of the cycle $a \cdot \sum_{0 \leqslant u < \ell}
  (-1)^{mu}t_ms_{m-1}t_{m-1}^u(x_{i_1} \otimes \dots \otimes
  x_{i_m})$, and the remaining homology groups are zero.
\end{enumerate}
\end{lemma}

\begin{proof}Let $D_*$ be the chain complex given by the quotient of the chain complex associated with the simplicial $\kappa$-module $\HH(A/\kappa;[\omega])[-]$ by the subcomplex of degenerate simplices. We recall that the canonical projection induces an isomorphism of $\HH_q(A/\kappa;[\omega])$ onto $H_q(D_*)$; see for example~\cite[Theorem~8.3.8]{weibel1}. We evaluate the chain complex $D_*$ in the three cases~(1)--(3). 

First, in the case~(1), $D_0$ is the free $\kappa$-module generated by $1$ and $D_q$ is zero, for $q > 0$. This proves statement~(1).

Next, in the case~(2), let $C_{\ell}$ be the cyclic group of order $\ell$, and let $\tau$ be a generator. We define $D_*'$ to be the chain complex with $D_q'= \kappa[C_{\ell}]$, if $q = m-1$ or $q = m$, and zero, otherwise, and with differential $d' \colon D_m' \to D_{m-1}'$ given by multiplication by $1- \tau$. Then the map $\alpha \colon D_*' \to D_*$ defined by
$$\begin{aligned}
\alpha_{m-1}(\tau^u) & = 
(-1)^{(m-1)u} t_{m-1}^u(x_{i_0} \otimes \dots \otimes x_{i_m}) \cr
\alpha_m(\tau^u) & = 
(-1)^{(m-1)u}t_ms_{m-1}t_{m-1}^u(x_{i_0} \otimes \dots \otimes x_{i_m}) \cr
\end{aligned}$$
is an isomorphism of chain complexes, since $(m-1)\ell$ is even. Now, the homology groups
$H_{m-1}(D_*')$ and $H_m(D_*')$ are free $\kappa$-modules of rank $1$
generated by the class of $1$ and the norm element $N = 1 + \tau +
\dots + \tau^{\ell-1}$, respectively. This proves the statement~(2). 

Finally, in the case~(3), let $C_{\ell}$ be the cyclic group of order
$\ell$, and let $\tau$ be a generator. 
%We note that, since $\ell$ is odd, $\tau^2$ is also a generator. 
We define $D_*''$ to be the chain complex with $D_q''=
\kappa[C_{\ell}]$, if $q = m-1$ or $q = m$, and zero, otherwise, and
with differential $d'' \colon D_m'' \to D_{m-1}''$ given by
multiplication by $1 + \tau$. Then the map
$\beta \colon D_*'' \to D_*$ defined by 
$$\begin{aligned}
\beta_{m-1}(\tau^u) & = 
(-1)^{mu}t_{m-1}^u(x_{i_0} \otimes \dots \otimes x_{i_m}) \cr
\beta_m(\tau^u) & = 
(-1)^{mu}t_ms_{m-1}t_{m-1}^u(x_{i_0} \otimes \dots \otimes x_{i_m}) \cr
\end{aligned}$$
is an isomorphism of chain complexes, since $m$ is even. Hence, to
prove statement~(3), it suffices to show that the following sequence
of $\kappa$-modules is exact.
$$0 \to \kappa[2] \xto{N} \kappa[C_{\ell}] \xto{1+\tau}
\kappa[C_{\ell}] \xto{\bar{\epsilon}} \kappa/2\kappa \to 0.$$
To this end, we consider the following commutative diagram with exact
rows.
$$\xymatrix{
{ 0 } \ar[r] &
{ I[C_{\ell}] } \ar[r] \ar[d]^{1 + \tau} &
{ \kappa[C_{\ell}] } \ar[r]^(.55){\epsilon} \ar[d]^{1 + \tau} &
{ \kappa } \ar[r] \ar[d]^{2} &
{ 0 } \cr
{ 0 } \ar[r] &
{ I[C_{\ell}] } \ar[r] &
{ \kappa[C_{\ell}] } \ar[r]^(.55){\epsilon} &
{ \kappa } \ar[r] &
{ 0 } \cr
}$$
The augmentation ideal $I[C_{\ell}]$ is equal to the
sub-$k[C_{\ell}]$-module generated by $1-\tau$. Since $\ell$ is odd,
$\tau^2$ is a generator of $C_{\ell}$, and hence, $1-\tau^2 =
(1+\tau)(1-\tau)$ is a generator of $I[C_{\ell}]$. This shows that the
left-hand vertical map $1+\tau$ is an isomorphism. Finally, the
following diagram commutes.
$$\xymatrix{
{ \kappa[2] } \ar@{=}[r] \ar[d]^{N} &
{ \kappa[2] } \ar@{^{(}->}[d] \cr
{ \kappa[C_{\ell}] } \ar[r]^{\epsilon} &
{ \kappa } \cr
}$$
Indeed, $\epsilon \circ N$ is equal to multiplication by $\ell$ which
is congruent to $1$ modulo $2$. This shows that the sequence in
question is exact. Statement~(3) follows. 
\end{proof}

\begin{remark}\label{productremark}For $\kappa$ a field of
characteristic zero, the Hochschild homology of the $\kappa$-algebra
$A$ in Lemma~\ref{summandhomology} was first evaluated by
Lindenstrauss~\cite[Theorem~3.1]{lindenstrauss} who also determined
the product structure of the graded $\kappa$-algebra $\HH_*(A/\kappa)$.
\end{remark}

\begin{proof}[Proof of Theorem~\ref{hhtheorem}]We let $\omega$ be the
word $(x_1,\dots,x_q)$ and consider the following compotision of the
map of the statement and the projection onto the summand $[\omega]$. 
$$K_q(A) \to
\HH_q(A) \to
\HH_q(A/\kappa) \xto{\pr_{[\omega]}}
\HH_q(A/\kappa,[\omega])$$
The Dennis trace map is a map of graded rings and takes the symbol
$\{1+x_i\}$ to the Hochschild homology class $d\log (1+x_i)$
represented by the cycle $1 \otimes x_i - x_i \otimes x_i$; see for example~\cite[Corollary~6.4.1]{gh},
\cite[Proposition~2.3.1]{hm4},
and~\cite[Proposition~1.4.5]{h}. Hence, 
$\{1+x_1,\dots,1+x_q\}$ is mapped to $d\log(1+x_1)
\dots  d\log(1+x_q)$. The product on Hochschild homology is given by
the shuffle product $*$, and moreover,
$$\pr_{[\omega]}(d\log(1+x_1)* \dots * d\log(1+x_q)) 
= \pr_{[\omega]}((1\otimes x_1)* \dots *(1\otimes x_q))$$
since summands that include a factor $x_i \otimes x_i$ are annihilated
by $\pr_{[\omega]}$. Now,
$$(1 \otimes x_1) * \dots * (1 \otimes x_q) = \sum_{\sigma}
\operatorname{sgn}(\sigma) 1 \otimes x_{\sigma(1)} \otimes
\dots \otimes x_{\sigma(q)},$$
where the sum ranges over all permutations of $\{1,2,\dots,q\}$, and
hence,
$$\pr_{[\omega]}((1 \otimes x_1) * \dots * (1 \otimes x_q)) =
\sum_{\tau} \operatorname{sgn}(\tau) 1 \otimes x_{\tau(1)}
\otimes \dots \otimes x_{\tau(q)},$$
where the sum range over all cyclic permutations of
$\{1,2, \dots, q\}$. By Lemma~\ref{summandhomology}~(2), this class is the
generator of $\HH_q(A/\kappa;[\omega])$. The theorem follows. 
\end{proof}

\section{Proof of Theorem~\ref{maintheorem}}

In this section, we prove Theorem~\ref{maintheorem} of the
introduction and a number of generalizations of this result.

\begin{proof}[Proof of Theorem~\ref{maintheorem}]It suffices to show
that for every maximal ideal $\mathfrak{m} \subset R$, the local ring 
$R_{\mathfrak{m}}$ is regular. The assumption that $R$ is
$K_{d+1}$-regular implies by~\cite[Theorem~2.1]{vorst1}
and~\cite[Corollary~2.1]{vorst} that the local ring $R_{\mathfrak{m}}$ is
$K_q$-regular for all $q \leqslant d+1$. The local ring
$R_{\mathfrak{m}}$ has dimension $d_{\mathfrak{m}} \leqslant
d$. We first argue that we may assume that $d_{\mathfrak{m}} =
d$. Let $I \subset R$ be the intersection of the minimal prime ideals 
$\mathfrak{p}_1,\dots,\mathfrak{p}_n \subset R$ that are not contained
in $\mathfrak{m}$. We claim that $\mathfrak{m} + I = R$. For if not, 
the ideal $\mathfrak{m} + I$ would be contained in a maximal ideal of
$R$ which necessarily would be $\mathfrak{m}$. Now, for each
$1\leqslant i \leqslant n$, we choose $y_i \in \mathfrak{p}_i$ with
$y_i \notin \mathfrak{m}$. Then $y = y_1 \dots y_n$ is in $I$, but not
in $\mathfrak{m}$. The claim follows. Now, by the Chinese remainder
theorem, there exists $r \in R$ such that $r \equiv 1 \mod
\mathfrak{m}$ and $r \equiv 0 \mod I$. We define $R' = R[1/r]$ and
$\mathfrak{m}' = \mathfrak{m}R'$. Then $\mathfrak{m}' \subset R'$ is a
maximal ideal, since $R'/\mathfrak{m}' = (R/\mathfrak{m})[1/r] =
R/\mathfrak{m}$, and the canonical map $R_{\mathfrak{m}} \to
R_{\mathfrak{m}'}'$ is an isomorphism. Moreover, the $k$-algebra $R'$
is of finite type, and since every minimal prime ideal of $R'$ is
contained in $\mathfrak{m}'$, we have $\dim R' = \dim
R_{\mathfrak{m}'}' = d_{\mathfrak{m}}$. Therefore, we may assume that
$d = d_{\mathfrak{m}}$. Hence, Theorem~\ref{ktheorem} shows that 
$$K_{d_{\mathfrak{m}}+1}(R_{\mathfrak{m}})/pK_{d_{\mathfrak{m}}+1}(R_{\mathfrak{m}})
= 0.$$
We choose a set of generators $x_1,\dots,x_r$ of the maximal ideal of
the local ring $R_{\mathfrak{m}}$. Then $r \geqslant d_{\mathfrak{m}}$
with equality if and only if $R_{\mathfrak{m}}$ is
regular. By~\cite[Theorem~28.3]{matsumura}, we may choose a $k$-algebra
section of the canonical projection
$R_{\mathfrak{m}}/\mathfrak{m}^2R_{\mathfrak{m}} \to R/\mathfrak{m} =
\kappa$. These choices give rise to a $k$-algebra isomorphism
$$A = \kappa[x_1,\dots,x_r]/(x_ix_j \mid 1 \leqslant i \leqslant j
\leqslant r) \xto{\sim}
R_{\mathfrak{m}}/\mathfrak{m}^2R_{\mathfrak{m}}.$$
Hence, Theorem~\ref{hhtheorem} shows that for all $1 \leqslant q
\leqslant r$, the symbol
$$\{1+x_1,\dots,1+x_q\} \in
K_q(R_{\mathfrak{m}})/pK_q(R_{\mathfrak{m}})$$
has non-trivial image in $K_q(A)/pK_q(A)$, and therefore, is
non-zero. Since the group
$K_{d_{\mathfrak{m}}+1}(R_{\mathfrak{m}})/pK_{d_{\mathfrak{m}}+1}(R_{\mathfrak{m}})$
is zero, we conclude that $r \leqslant d_{\mathfrak{m}}$ which shows that
$R_{\mathfrak{m}}$ is a regular local ring. This completes the proof.
\end{proof}

\begin{theorem}\label{maintheoremplusr}Let $k$ be a field of
positive characteristic $p$ that is finitely generated over an
infinite perfect subfield $k'$, and assume that strong resolution of
singularities holds over $k'$. Let $R$ be a localization of a
$d$-dimensional commutative $k$-algebra of finite type and suppose
that $R$ is $K_{d+r+1}$-regular where $r$ is the transcendence degree
of $k$ over $k'$. Then $R$ is a regular ring.
\end{theorem}

\begin{proof}We can write $R$ as the localization $f \colon R' \to S^{-1}R' = R$
of a $(d+r)$-dimensional commutative $k'$-algebra $R'$ of finite type
with respect to a multiplicative subset $S \subset R'$. Let $\mathfrak{p}
\subset R$ be a prime ideal. Then, by~\cite[Theorem~2.1]{vorst1}, the local
ring $R_{\mathfrak{p}}$ again is $K_{d+r+1}$-regular. Now, let $\mathfrak{p}' =
f^{-1}(\mathfrak{p}) \subset R'$. Then the map $f$ induces an
isomorphism of $R'_{\mathfrak{p}'}$ onto
$R_{\mathfrak{p}}$. Therefore, we conclude from
Theorem~\ref{maintheorem} that $R_{\mathfrak{p}}$ is a regular
ring. This proves that $R$ is a regular ring as stated.
\end{proof}

\begin{theorem}\label{maintheoremplus}Let $p$ be a prime number and
assume that strong resolution of singularities holds over all infinite
perfect fields of characteristic $p$. Let $k$ be any field that
contains an infinite perfect subfield of characteristic $p$, let $R$
be a commutative $k$-algebra essentially of finite type, and assume
that $R$ is $K_q$-regular for all $q$. Then $R$ is a regular ring.
\end{theorem}

\begin{proof}We can write $R$ as a localization of $R' \otimes_{k'}k$
where $k'$ is a finitely generated field that contains an infinite
perfect subfield and where $R'$ is a commutative $k'$-algebra of
finite type. Then we can write $R$ as the filtered colimit
$$R = \colim_{\alpha} R' \otimes_{k'}k_{\alpha}$$
where $k_{\alpha}$ runs through the finitely generated extensions of $k'$
contained in $k$. It follows from Theorem~\ref{maintheoremplusr} that
the rings $R' \otimes_{k'}k_{\alpha}$ are all regular. Therefore the
ring $R$ is regular by~\cite[Prop.~IV.5.13.7]{ega}.
\end{proof}

\providecommand{\bysame}{\leavevmode\hbox to3em{\hrulefill}\thinspace}
\providecommand{\MR}{\relax\ifhmode\unskip\space\fi MR }
% \MRhref is called by the amsart/book/proc definition of \MR.
\providecommand{\MRhref}[2]{%
  \href{http://www.ams.org/mathscinet-getitem?mr=#1}{#2}
}
\providecommand{\href}[2]{#2}


\begin{thebibliography}{10}

\bibitem{bass1}
H.~Bass, \emph{Some problems in {``classical''} algebraic {$K$}-theory},
  Algebraic {$K$}-theory, {II}: {``Classical''} algebraic {$K$}-theory and
  connections with arithmetic (Proc. Conf., Battelle Memorial Inst., Seattle
  Wash., 1972), Lecture Notes in Math., vol. 342, Springer-Verlag, Berlin,
  1973, pp.~3--73.

\bibitem{cortinashaesemeyerweibel}
G.~Corti\~{n}as, C.~Haesemeyer, and C.~A. Weibel, \emph{{$K$}-regularity,
  {$\operatorname{cdh}$}-fibrant {H}ochschild homology, and a conjecture of
  {V}orst}, J. Amer. Math. Soc. \textbf{21} (2008), 547--561.

\bibitem{gh}
T.~Geisser and L.~Hesselholt, \emph{Topological cyclic homology of schemes},
  {$K$}-theory (Seattle, 1997), Proc. Symp. Pure Math., vol.~67, 1999,
  pp.~41--87.

\bibitem{geisserlevine}
T.~Geisser and M.~Levine, \emph{The {$K$}-theory of fields in characteristic
  {$p$}}, Invent. Math. \textbf{139} (2000), 459--493.

\bibitem{ega}
A.~Grothendieck and J.~Dieudonn{\'{e}}, \emph{{\'E}l{\'e}ments de
  g{\'e}om{\'e}trie alg{\'e}brique}, Inst. Hautes {\'{E}}tudes Sci. Publ. Math.
  \textbf{4, 8, 11, 17, 20, 24, 28, 32} (1960--1967).

\bibitem{haesemeyer}
C.~Haesemeyer, \emph{Descent properties of homotopy {$K$}-theory}, Duke Math.
  J. \textbf{125} (2004), 589--620.

\bibitem{h}
L.~Hesselholt, \emph{On the $p$-typical curves in {Q}uillen's {$K$}-theory},
  Acta Math. \textbf{177} (1997), 1--53.

\bibitem{hm4}
L.~Hesselholt and I.~Madsen, \emph{On the {$K$}-theory of local fields}, Ann.
  of Math. \textbf{158} (2003), 1--113.

\bibitem{kratzer}
C.~Kratzer, \emph{{$\lambda$}-structure en {$K$}-th{\'{e}}orie
  alg{\'{e}}brique}, Comment. Math. Helv. \textbf{55} (1980), 233--254.

\bibitem{lindenstrauss}
A.~Lindenstrauss, \emph{The {H}ochschild homology of fatpoints}, Israel J.
  Math. \textbf{133} (2003), 177--188.

\bibitem{matsumura}
H.~Matsumura, \emph{Commutative ring theory}, Cambridge studies in advanced
  mathematics, vol.~8, Cambridge University Press, 1986.

\bibitem{quillen}
D.~Quillen, \emph{Higher algebraic {$K$}-theory {I}}, Algebraic {$K$}-theory I:
  Higher {$K$}-theories (Battelle Memorial Inst., Seattle, Washington, 1972),
  Lecture Notes in Math., vol. 341, Springer-Verlag, New York, 1973.

\bibitem{vorst}
T.~Vorst, \emph{Localization of the {$K$}-theory of polynomial extensions},
  Math. Ann. \textbf{244} (1979), 33--53.

\bibitem{vorst1}
\bysame, \emph{A survey on the {$K$}-theory of polynomial extensions},
  Algebraic {$K$}-theory, number theory, Geometry and Analysis (Proceedings,
  Bielefeld 1982), Lecture Notes in Math., vol. 1046, Springer-Verlag, Berlin,
  1984, pp.~422--441.

\bibitem{weibel4}
C.~A. Weibel, \emph{Homotopy algebraic {$K$}-theory}, Algebraic {$K$}-theory
  and number theory (Honolulu, HI, 1987), Contemp. Math., vol.~83, Amer. Math.
  Soc., Providence, RI, 1989, pp.~461--488.

\bibitem{weibel1}
\bysame, \emph{An introduction to homological algebra}, Cambridge Studies in
  Advanced Mathematics, vol.~38, Cambridge University Press, Cambridge, 1994.

\end{thebibliography}
\end{document}